\documentclass[10pt]{article}
\usepackage[all]{xy}
\usepackage{amsfonts,amsmath,oldgerm,amssymb,amscd}
\newcommand{\ra}{\rightarrow}
\newcommand{\lra}{\longrightarrow}

\newcommand{\Lby}[1]{\stackrel{#1}{-\!\!-\!\!\!\lra}}

\newcommand{\ol}{\overline}		\newcommand{\wt}{\widetilde}

\newtheorem{theorem}{Theorem}[section]
\newtheorem{proposition}[theorem]{Proposition}
\newtheorem{lemma}[theorem]{Lemma}

\newcommand{\gt}{\theta}

	\newcommand{\BF}{\mbox{$\mathbb F$}}

\newcommand{\BQ}{\mbox{$\mathbb Q$}}

	\newcommand{\BZ}{\mbox{$\mathbb Z$}}

\newcommand{\op}{\mbox{$\oplus$}}	
	
	\newcommand{\Hom}{\mbox{\rm Hom}}
  	
\newcommand{\Um}{\mbox{\rm Um}}		
\newcommand{\GL}{\mbox{\rm GL}}		
\newcommand{\Aut}{\mbox{\rm Aut}}

\begin{document} 

\begin{center}
{\Large \bf A note on cancellation of projective modules}
\\
\vspace{.2in} {\large Alpesh M. Dhorajia and Manoj K. Keshari       }\\
\vspace{.1in} {\small Department of Mathematics, IIT Mumbai, Mumbai -
400076, India \\ (alpesh,keshari)@math.iitb.ac.in}
\end{center}

\section{Introduction}

Let $A$ be a ring of dimension $d$ and let $P$ be a projective
$A$-module of rank $d$. Assume that if $R$ is a finite extension of
$A$ then $R^d$ is cancellative. Then it is proved in (\cite{K},
Theorem 3.6) that $P$ is also cancellative. In other words, if
$\GL_{d+1}(R)$ acts transitively on $\Um_{d+1}(R)$ for every finite
extension $R$ of $A$, then $\Aut (A\op P)$ acts transitively on
$\Um(A\op P)$.

We will generalize  the above result as follows (\ref{c2}). 

\begin{theorem}\label{m}
Let $A$ be a ring of dimension $d$ and let $P$ be a projective
$A$-module of rank $d$. Assume that if $R$ is a finite extension of
$A$ then $E_{d+1}(R)$ acts transitively on $\Um_{d+1}(R)$. Then
$E(A\op P)$ acts transitively on $\Um(A\op P)$. 
\end{theorem}

If $A$ is an affine algebra of dimension $d$ over $\BZ$ then
Vaserstein \cite{V} proved that $E_{d+1}(A)$ acts transitively on
$\Um_{d+1}(A)$. As a consequence of (\ref{m}) we get another proof of
the following result of Mohan Kumar, Murthy and Roy (\cite{MMR},
Theorem 2.4) that if $P$ is a projective $A$-module of rank $d$, then
$E(A\op P)$ acts transitively on $\Um(A\op P)$.

Let $A$ be a smooth affine algebra of dimension $d$ over an
algebraically closed field $\ol k$. Assume that if characteristic of
$\ol k$ is $p>0$, then $p\geq d$. Recently Fasel, Rao and Swan
(\cite{frs}, Theorem 7.3) proved that stably free $A$-modules of rank
$d-1$ are free, thus answering an old question of Suslin. Infact if
$d\geq 4$, then they proved that $A$ being normal suffices. In view of
their result, a natural question arises: Let $P$ be a projective
$A$-module of rank $d-1$. Is $P$ cancellative? We answer this question
in affirmative when $\ol k=\ol \BF_p$. More precisely, we prove the
following result (\ref{c3}).

\begin{theorem}\label{m2}
Let $A$ be an affine algebra of dimension $d\geq 4$ over $\ol \BF_p$,
where $p\geq d$. Then every projective $A$-module of rank $d-1$ is
cancellative.
\end{theorem} 

Finally, we will prove the following result (\ref{g5}).  Gubeladze
proved this result (\cite{G2}, \cite{G3}) in case $P$ is free.

\begin{theorem}
Let $M\subset \BQ_+^r$ be a seminormal monoid such that $M\subset
\BQ_+^r$ is an integral extension.  Let $R$ be a ring of dimension $d$
and let $P$ be a projective $R[M]$-module of rank $n$. Then $E(R[M]\op
P)$ acts transitively on $\Um(R[M]\op P)$ whenever $n\geq $ max
$(2,d+1)$.
\end{theorem}

\section{Preliminaries}

All the rings are assumed to be commutative Noetherian and all the
modules are finitely generated. 

Let $A$ be a ring and let $M$ be an $A$-module. We say that $m\in M$
is {\it unimodular} if there exists $\phi\in M^*$ such that
$\phi(m)=1$. The set of all unimodular elements of $M$ will be
denoted by $\Um(M)$.  We denote by $\Aut_A(M)$, the group of all
$A$-automorphism of $M$.  For an ideal $J$ of $A$, we denote by
$\Aut_A(M,J)$, the kernel of the natural homomorphism
$\Aut_A(M)\rightarrow \Aut_A(M/JM)$.

We denote by $EL^1(A\oplus M,J)$, the subgroup of $\Aut_A(A\oplus M)$
generated by all the automorphisms $\Delta_{a\varphi}=\left(
\begin{smallmatrix}
 1 & a\varphi\\
0 & id_M
\end{smallmatrix}  \right)
$ and $\Gamma_{m}=\left(\begin{smallmatrix}
1&0\\
m&id_M 
\end{smallmatrix}\right)$
with $a\in J,\varphi \in M^*$ and $m \in M$.  We will write $EL^1(A\op
M)$ for $EL^1(A\op M,A)$.

We denote by $\Um^1(A\oplus M, J)$, the set of all $(a,m) \in
\Um(A\oplus M)$ such that $a \in 1 + J$ and by $\Um(A\oplus M,J)$, the
set of all $(a,m) \in \Um^1(A\oplus M,J)$ with $m \in JM$. We will
write $\Um^1_r(A,J)$ for $\Um^1(A\op A^{r-1},J)$ and $\Um_r(A,J)$ for
$\Um(A\op A^{r-1},J)$.

Let $M$ be an $A$-module. Let $p\in M$ and $\varphi \in M^*$ be such
that $\varphi(m)=0$. Let $\varphi_p \in End(M)$ be defined as
$\varphi_p(q)=\varphi(q)p$. Then $1+\varphi_p$ is a unipotent
automorphism of $M$. An automorphism of $M$ of the form $1+\varphi_p$
is called a {\it transvection} of $M$ if either $p\in \Um(M)$ or
$\varphi \in \Um(M^*)$. We denote by $E(M)$, the subgroup of $\Aut(M)$
generated by all the transvections of $M$.

The following result is due to Bak, Basu and Rao (\cite{bbr}, Theorem
3.10). In \cite{AK}, we have proved results for $EL(A\op P)$. Due to
this result, we can use $E(A\op P)$ everywhere.

\begin{theorem}\label{bb}
Let $A$ be a ring and let $P$ be a projective $A$-module of rank $\geq
2$. Then $EL(A\op P)=E(A\op P)$.
\end{theorem}

\begin{remark}\label{rem}
Using (\ref{bb}), it is easy to see that if $I$ is any ideal of $A$,
then the natural map $E(A\op P) \ra E((A\op P)/I(A\op P))$ is surjective.
\end{remark}

The following is a classical result due to Bass \cite{Bass}.

\begin{theorem}\label{bass}
Let $A$ be a ring and let $P$ be a projective $A$-module of rank
$>\dim A$. Then $E(A\op P)$ acts transitively on $\Um(A\op P)$.
\end{theorem}

The following result is due to Lindel (\cite{L}, Lemma 1.1).

\begin{lemma}\label{3.1}
Let $A$ be a  ring and let $P$ be a projective $A$-module of rank $r$.
Then there exists $s\in A$ such that the following holds:

$(i)$ $P_s$ is free, 

$(ii)$ there exists $p_1,\ldots,p_r\in P$ and
$\phi_1,\ldots,\phi_r\in \Hom(P,A)$ such that
$(\phi_i(p_j))=$ diagonal $(s,\ldots,s)$,

$(iii)$ $sP\subset p_1A+\ldots+ p_rA$,

$(iv)$ the image of $s$ in $A_{red}$ is a non-zero-divisor and

$(v)$ $(0:sA)=(0:s^2A)$.
\end{lemma}

The following two results are from (\cite{AK}, Lemma 3.1 and Lemma 3.10).

\begin{lemma}\label{3.10}
Let $A$ be a ring and let $P$ be a projective $A$-module. Let ``bar''
denote reduction modulo the nil radical of $A$. If $E(\ol A\op \ol P)$
acts transitively on $\Um(\ol A\op \ol P)$, then $E( A\op P)$ acts
transitively on $\Um(A\op P)$.
\end{lemma}

\begin{lemma}\label{l1}
Let $A$ be a ring and let $P$ be a projective $A$-module of rank
$r$. Choose $s\in A$, $p_1,\ldots,p_r\in P$ and
$\varphi_1,\ldots,\varphi_r\in P^*$ satisfying the properties of
(\ref{3.1}). Let $(a,p)\in \Um(A\op P,sA)$ with $p=c_1p_1+\ldots
+c_rp_r$, where $c_i \in sA$ for $i=1,\ldots,r$. Assume there exists
$\phi\in E_{r+1}^1(R,sA)$ such that
$\phi(a,c_1,\ldots,c_r)=(1,0,\ldots,0)$. Then there exists $\Phi\in
E(A\op P)$ such that $\Phi(a,p)=(1,0)$.
\end{lemma}

The following result is due to Mohan Kumar, Murthy and Roy (\cite{MMR},
Theorem 2.4).

\begin{theorem}\label{mmr}
Let $A$ be an affine algebra of dimension $d\geq 2$ over $\ol \BF_p$. Let
$P$ be a projective $A$-module of rank $d$. Then $E(A\op P)$ acts
transitively on $\Um(A\op P)$.
\end{theorem}

The following result (\cite{K}, Theorem 3.8) is very crucial for the
proof of (\ref{c3}).

\begin{theorem}\label{MKK}
Let $A$ be an affine algebra of dimension $d$ over $\ol
\BF_p$. Assume that if $R$ is a finite extension of $A$ then
$R^{d-1}$ is cancellative. Then every projective $A$-module of rank
$d-1$ is cancellative.
\end{theorem}

We end this section with a result  due to Fasel, Rao and
Swan (\cite{frs}, Corollary 7.4).

\begin{proposition}\label{7.4}
Let $R$ be an affine algebra of dimension $d\geq 4$ over an
algebraically closed field $\ol k$. Assume that if characteristic of
$\ol k$ is $p>0$, then $p\geq d$. Let $J$ be the ideal defining the
singular locus of $R$. Then $\GL_d(R)$ acts transitively on $\Um_d(R,J)$.
\end{proposition}


\section{Main Theorem} 

In this section, we prove our main result. 

Let $A$ be a ring and $I$ an ideal of $A$. For an integer $n\geq 3$,
define $E_n(I)$ as the subgroup of $E_n(A)$ generated by
$E_{ij}(a)=Id+ae_{ij}$, where $a\in I$, $1\leq i\ne j\leq n$ and only
non-zero entry of the matrix $ae_{ij}$ is $a$ at the $(i,j)$th place.

Consider the cartesian square
$$\xymatrix{ A(I)\ar [r]^{p_1} \ar [d]_{p_2} & A
\ar [d]^{j_1} \\ A \ar [r]_{j_2} & A/I }$$

The relative group $E_n(A,I)$ is defined in \cite{S} by the exact
sequence $$1\ra E_n(A,I)\ra E_n(A(I))\Lby {E_n(p_1)} E_n(A) \ra 1$$
and it is shown (\cite{S}, Proposition 2.2) that $E_n(A,I)$ is
isomorphic to the kernel of the natural map $E_n(A)\ra E_n(A/I)$.
Further, $E_n(A,I)$ is the normal closure of $E_n(I)$ in $E_n(A)$
\cite{Bass1}.

\begin{lemma}\label{P1}
Let $R$ be a ring and $I$ an ideal of $R$. If $n\geq 3$, then
$E_n(R,I^2)\subset E_n(I)$.
\end{lemma}

\begin{proof} 
It is enough to show that if $\beta = E_{ij}(z)\in E_n(I^2)$ and
$\alpha = E_{kl}(z')\in E_n(R)$ then $\alpha\beta\alpha^{-1}\in
E_n(I)$. Note that $e_{ij}e_{kl}=e_{il}$ if $j=k$ and $0$ if $j\ne k$.

First assume that $i\ne l$ and hence $(i,j)\neq (l,k)$.  Then
$\alpha\beta\alpha^{-1} = (Id+z'e_{kl})(Id+ze_{ij})(Id-z'e_{kl}) =
(Id+z'e_{kl}+ze_{ij})(Id-z'e_{kl})=Id+z'e_{kl}+ze_{ij}-
z'e_{kl}-z'^2e_{kl}e_{kl}- z'ze_{ij}e_{kl}
=Id+ze_{ij}-z'ze_{ij}e_{kl}$.  If $j=k$, then $\alpha\beta\alpha^{-1}
=Id+ze_{ij}-z'ze_{il}=(Id+ze_{ij})(Id-z'ze_{il}) \in E(I)$. Further, if
$j\ne k$, then $\alpha\beta\alpha^{-1} =I+ze_{ij} \in E(I)$. This
proves that if $i\ne l$ then $\alpha\beta\alpha^{-1} \in E(I)$.
Similarly we can prove that if $j\ne k$ then $\alpha\beta\alpha^{-1}
\in E(I)$.

Now assume that $(i,j)=(l,k)$. Choose
$r\leq n$ different from $i,j$ and write $z=
a_1b_1+a_2b_2+\ldots+a_sb_s$, where $a_i,b_i\in I$. Now we can write
$$\beta = E_{ij}(z) = \prod_{t=1}^s E_{ij}(a_tb_t)=
\prod_{t=1}^{s}[E_{ir}(a_t),E_{rj}(b_t)]$$
and
$$\alpha\beta\alpha^{-1} = \prod_{t=1}^{s}[\alpha
E_{ir}(a_t)\alpha^{-1}, \alpha E_{rj}(b_t)\alpha^{-1}]\in E_n(I)$$
$\hfill \square$
\end{proof}

\begin{lemma}\label{P2}
Let $R$ be a ring and $I$ an ideal of $R$. If $n\geq 3$, then
$E_n(I)\subset E_n^1(R,I)$ and hence $E_n(R,I^2)\subset E_{n}^1(R, I)$.
\end{lemma}

\begin{proof}
 Let $E_{ij}(x)\in E_{n}(I)$, where $x\in I$. If $i=1$ or $j=1$, then
$E_{ij}(x)\in E_n^1(R,I)$. Assume $i\ne 1$ and $j\ne 1$.  Then $E_{ij}(x) =
E_{i1}(1)E_{1j}(x)E_{i1}(-1)E_{1j}(-x)\in E_{n}^1(R, I)$.  $\hfill
\square$
\end{proof}

\begin{lemma}\label{c1}
Let $A$ be a ring and let $P$ be a projective $A$-module of rank $r$.
Choose $s\in A$ satisfying the conditions in (\ref{3.1}). Assume that
if $R=A[X]/(X^2-s^2X)$ then $E_{r+1}(R)$ acts
transitively on $\Um_{r+1}(R)$. Then $E(A\op P)$ acts transitively on
$\Um(A\op P,s^2A)$.
\end{lemma}

\begin{proof}
Without loss of generality, we may assume that $A$ is reduced.
By
(\ref{3.1}), there exist $p_1,\ldots,p_r\in P$ and
$\phi_1,\ldots,\phi_r \in \Hom(P,A)$ such that $P_s$ is free,
$(\phi_i(p_j))=$ diagonal $(s,\ldots,s)$, $sP\subset p_1A+\ldots+
p_rA$ and $s$ is a non-zerodivisor.

Let $(a,p)\in \Um(A\op P,s^2A)$. Replacing $p$ by $p-ap$, we may assume
that $p\in s^3P$. Since $sP\subset \sum_1^r Ap_i$, we get
$p=f_1p_1+\ldots +f_rp_r$ for some $f_i \in s^2A$. Note that
$v=(a,f_1,\ldots,f_r) \in \Um_{r+1}(A,s^2A)$.

Consider the following cartesian square
$$\xymatrix{ R\ar [r]^{p_1} \ar [d]_{p_2} & A
\ar [d]^{j_1} \\ A \ar [r]_{j_2} & A/(s^2) }$$

Patching unimodular rows $(a,f_1,\ldots,f_r)$ and $(1,0,\ldots,0)$
over $A/s^2A$, we get a unimodular row $(c_0,c_1,\ldots,c_r)\in
\Um_{r+1}(R)$. Since $E_{r+1}(R)$ acts transitively on $\Um_{r+1}(R)$,
there exists $\Theta \in E_{r+1}(R)$ such that
$(c_0,c_1,\ldots,c_r)\Theta=(1,0,\ldots,0)$. The projections of this
equation gives
$$(f,f_1,\ldots,f_r)\Psi=(1,0,\ldots,0)\;\text{ and}\;\;
(1,0,\ldots,0)\wt \Psi=(1,0,\ldots,0)$$ 
where
$\Psi,\wt \Psi\in E_{r+1}(A)$ such that $\Psi=\wt\Psi$ modulo
$(s^2)$.  Hence $(f,f_1,\ldots,f_r)\Psi\,\wt\Psi^{-1}=(1,0,\ldots,0)$,
where $\Psi\wt\Psi^{-1}=\Delta\in E_{r+1}(A,s^2A)$.

By (\ref{P2}), $\Delta \in E_{r+1}^1(A,sA)$. Hence by (\ref{l1}),
there exists  $\Theta \in E(A\op P)$ such that
$(a,p)\Theta=(1,0)$. This completes the proof.  $\hfill \square$
\end{proof}

\begin{theorem}\label{c2}
Let $A$ be a ring of dimension $d$ and let $P$ be a projective
$A$-module of rank $d$.  Assume that if $R$ is a finite extension of
$A$ then $E_{d+1}(R)$ acts transitively on $\Um_{d+1}(R)$. Then
$E(A\op P)$ acts transitively on $\Um(A\op P)$.
\end{theorem}

\begin{proof}
Let $(a,p)\in \Um(A\op P)$.  Choose $s\in A$ satisfying the conditions
in (\ref{3.1}). Let ``bar'' denote reduction modulo $s^2A$. Since $\dim
\ol A=d-1$, by (\ref{bass}), there exists $\sigma \in E(\ol A\op \ol
P)$ such that $(\ol a,\ol p)\sigma=(1,0)$. By (\ref{rem}), we can lift
$\sigma$ to $\theta\in E(A\op P)$. If $(a,p)\theta=(b,q)$, then
$(b,q)\in \Um(A\op P,s^2A)$. By (\ref{c1}), there exists $\theta_1\in
E(A\op P)$ such that $(b,q)\theta_1=(a,p)\gt \gt_1=(1,0)$. This proves
the result.  $\hfill \square$
\end{proof}

\begin{theorem}\label{c3}
Let $A$ be an affine algebra of dimension $d\geq 4$ over the field $\ol
\BF_p$, where $p\geq d$. Let $P$ be a projective $A$-module of rank
$d-1$. Then $P$ is cancellative.
\end{theorem}

\begin{proof}
By (\ref{MKK}), it is enough to show that if $R$ is any affine algebra
of dimension $d$ over $\ol \BF_p$, then $R^{d-1}$ is cancellative.
Let $v\in \Um_d(R)$ be any unimodular row of length $d$. It is enough
to show that there exists $\sigma \in \GL_d(R)$ such that
$v\sigma=e_1=(1,0,\ldots,0)$. Without loss of generality, we may
assume that $ R$ is reduced.

Let $J$ be the ideal of $R$ defining the singular locus of $R$. Since
$R$ is reduced, height of $J$ is $\geq 1$. Let ``bar'' denote
reduction modulo $J$. Then $\dim \ol R=d-1$. By (\ref{mmr}), there
exits $\sigma\in E_d(\ol R)$ such that $\ol v \sigma =e_1$.  By
(\ref{rem}), we can lift $\sigma$ to $\theta \in E_d(R)$.  We have
$v\theta=e_1$ modulo $J$. Applying (\ref{7.4}), we get $\gt\in
\GL_d(R)$ such that $v\sigma \gt =e_1$. Hence  $v$ is completable to an
invertible matrix, i.e. $R^{d-1}$ is cancellative. This
completes the proof. $\hfill \square$
\end{proof}

\section{Extension of Gubeladze's results}

In this section we extend some results of Gubeladze. We begin by
recalling three results due to Gubeladze \cite{G1}, (\cite{G2},
Theorem 8.1) and (\cite{G3}, Theorem 10.1) respectively. See \cite{G3}
for the definition of a monoid $M$ of $\Phi$-simplicial growth.

\begin{theorem}\label{g1}
Let $M$ be a commutative torsion-free seminormal and cancellative
monoid.  Then for any principal ideal domain $R$, projective modules
over $R[M]$ are free.
\end{theorem}

\begin{theorem}\label{g2}
Let $R$ be a ring of dimension $d$ and let $M\subset \BQ_+^r$ be a
submonoid such that $M\subset \BQ_+^r$ is an integral extension. Then
$E_n(R[M])$ acts transitively on $\Um_n(R[M])$ whenever $n\geq $ max
$(3,d+2)$.
\end{theorem}

\begin{theorem}\label{g3}
Let $R$ be a ring of dimension $d$ and let $M$ be a monoid of
$\Phi$-simplicial growth. Then $E_n(R[M])$ acts transitively on
$\Um_n(R[M])$ whenever $n\geq$ max $(3,d+2)$.
\end{theorem}

We will generalize above results as follows.

\begin{theorem}\label{g4}
Let $M$ be as in (\ref{g2}) or (\ref{g3}). Let $R$ be a ring of
dimension $d$ and let $P$ be a projective $R[M]$-module of rank
$n$. Assume that $S^{-1}P$ is free, where $S$ is the set of
non-zerodivisors of $R$. Then $E(R[M]\op P)$ acts transitively on
$\Um(R[M]\op P)$ whenever $n\geq $ max $(2,d+1)$.
\end{theorem}

\begin{proof}
By (\ref{3.10}), we may assume that the ring $A=R[M]$ is reduced.  We
will use induction on $d$. If $d=0$, then by assumption, projective
modules of constant rank over $R[M]$ are free. Hence we are done by
(\ref{g2}) and (\ref{g3}).

Assume $d>0$. By assumption $S^{-1}P$ is free. We can choose $s\in S$
such that $P_s$ is free and conditions of (\ref{3.1}) are satisfied.

Let $(a,p)\in \Um(A\op P)$ and let ``bar'' denote reduction modulo
$s^2A$. Since $\dim \ol R=d-1$, by induction hypothesis, there exists
$\phi\in E(\ol A\op \ol P)$ such that $(\ol a,\ol p)\phi =(1,0)$. Let
$\Phi \in E(A\op P)$ be a lift of $\phi$, by (\ref{rem}). Then
$(a,p)\Phi \in \Um(A\op P,s^2A)$. By Gubeladze's theorem in the free
case, $E_n(B[M])$ acts transitively on $\Um_n(B[M])$, where
$B=R[X]/(X^2-s^2X)$. Hence by (\ref{c1}), there exists $\Phi_1\in
E(A\op P)$ such that $(a,p)\Phi\Phi_1=(1,0)$. This completes the
proof.  $\hfill \square$
\end{proof}

Using (\ref{g1}, \ref{g4}), we get the following.

\begin{theorem}\label{g5}
Let $M$ be as in (\ref{g2}) or (\ref{g3}). Further assume that $M$ is
seminormal.  Let $R$ be a ring of dimension $d$ and let $P$ be a
projective $R[M]$-module of rank $n$. Then $E(R[M]\op P)$ acts
transitively on $\Um(R[M]\op P)$ whenever $n\geq $ max $(2,d+1)$.
\end{theorem}

{\bf Acknowledgement:} The first author would like to thank Professor
Nikolai Vavilov for useful discussion and second author would like to
thank Professor Bhatwadekar for pointing out a gap in an earlier
version.


{}

\end{document}